\documentclass[11pt,reqno]{amsart}

\parskip=\smallskipamount

\newtheorem{theorem}{Theorem}[section]
\newtheorem{lemma}[theorem]{Lemma}

\theoremstyle{definition}
\newtheorem{definition}[theorem]{Definition}

\newtheorem{remark}[theorem]{Remark}

\newcommand{\bea}{\begin{eqnarray*}}
\newcommand{\eea}{\end{eqnarray*}}

\numberwithin{equation}{section}

\begin{document}

\author{Erlend Forn\ae ss Wold}
\thanks{Part of this work was done during the international research program "Several Complex Variables and Complex Dynamics"
at the Center for Advanced Study at the Academy of Science and Letters in Oslo during the academic year 2016/2017. }

\title[Invariant Metrics]{Asymptotics of Invariant Metrics in the normal direction and a new characterisation of the unit disk}
%
%
\subjclass[2000]{}
\date{\today}

\maketitle

\begin{abstract}
We give improvements of estimates of invariant metrics in the normal direction on strictly pseudoconvex domains. 
Specifically we will give the second term in the expansion of the metrics.  This depends on an improved localisation 
result and estimates in the one variable case.  Finally we will give a new characterisation of the unit disk in $\mathbb C$
in terms of the asymptotic behaviour of quotients of invariant metrics.   
\end{abstract}

\section{Introduction}

In this section we state our main results, and in the next section we give some background.   
Let $z_i=x_{2i-1}+ix_{2i}$ denote the coordinates on $\mathbb C^n$ for 
$i=1,...,n$.
For a smooth real valued function $\phi$ we let $H_\phi(p)$ denote the real Hessian

\begin{equation}
H_\phi(p)(v)=\sum_{i,j}\frac{\partial^2\phi}{\partial x_i\partial x_j}(p)v_i\cdot v_j.
\end{equation}
If $\Omega\subset\mathbb C^n$ is a domain with a smooth defining 
function $\phi$ near $p\in b\Omega$ we let $\eta_p$ denote 
the outward normal vector $\eta_p=\nabla\phi(p)$, we set
$\tilde\eta_p=\frac{\eta_p}{\|\eta_p\|}$, we let $L_p$
denote the complex line through $p$ generated by $\eta_p$, and we
we let $\kappa_p$ denote the quantity 
\begin{equation}
\kappa_p:=\frac{H_\phi(p)(J\eta_p)}{\|\eta_p\|^3} 
\end{equation}
The quantity $\kappa_p$ is the curvature of the curve $b\Omega\cap L_p$ with 
respect to the direction $\eta_p$; in particular it is independent of any particular
choice of $\phi$.

\begin{theorem}\label{main-invariant}
Let $\Omega\subset\mathbb C^n$ be a bounded domain such that $\overline\Omega$
is a Stein compact, and assume that $p\in b\Omega$ is a strictly pseudoconvex 
boundary point of class $C^4$.   Then 
\begin{equation}\label{main-invariant-exp}
F_\Omega(p-\delta\tilde\eta_p,\tilde\eta_p)=\frac{1}{2\delta} + \frac{\kappa_p}{4} + O(\delta),
\end{equation} 
where $F$ denotes either the Carath\'{e}odory or the Kobayashi metric.
\end{theorem}

This result improves previous estimates due to Fu, Graham and Ma \cite{Fu},\cite{Graham},\cite{Graham}.

If $n=1$ the theorem holds also for the Bergman Kernel function on the diagonal, see Theorem \ref{plane-invariant}.  
The theorem follows from the corresponding result for $n=1$ to be proved in section three, together with the following localisation result

\begin{theorem}\label{main-localise}
Let $\Omega\subset\mathbb C^n$ be a bounded domain such that $\overline\Omega$
is a Stein compact, and assume that $p\in b\Omega$ is a strictly pseudoconvex 
boundary point of class $C^4$.  Let $\lambda>0$. Then 
\begin{equation}
F_{\Omega\cap (L_p\cap D_\lambda(p))}(p-\delta\tilde\eta_p,\tilde\eta_p) - F_\Omega(p-\delta\tilde\eta_p,\tilde\eta_p) = O(\delta).
\end{equation}
\end{theorem}

As a corollary to Theorem \ref{main-invariant} we get the following comparison theorem for invariant metrics.  

\begin{theorem}
Let $\Omega\subset\mathbb C^n$ be a bounded domain such that $\overline\Omega$
is a Stein compact, and assume that $p\in b\Omega$ is a strictly pseudoconvex 
boundary point of class $C^4$.  Then 
\begin{equation}
F^1_\Omega(p-\delta\tilde\eta_p,\tilde\eta_p) - F^2_\Omega(p-\delta\tilde\eta_p,\tilde\eta_p)=O(\delta),
\end{equation}
or equivalently 
\begin{equation}
\frac{F^1_\Omega(p-\delta\tilde\eta_p,\tilde\eta_p)}{F^2_\Omega(p-\delta\tilde\eta_p,\tilde\eta_p)}=1+O(\delta^2),
\end{equation}
where $Fj$ is any invariant metric. 
\end{theorem}

\begin{remark}
The only special property of the normal direction in the previous three results is that it is not tangential 
to the boundary; any non-tangential direction can be made normal by a linear change of coordinates.   
Furthermore, philosophically the estimates in the last two results should be better in directions close to tangential.  
This should be pursued further. 
\end{remark}

The comparison theorem just given continues to hold in the complex 
plane in the case of $C^2$-smooth domains.   The estimate is then even sharp, and so it gives a new characterisation of 
the unit disk, where all metrics coincide, in terms of asymptotics of quotients of invariant metrics (see also Theorem \ref{equivalence} below).

\begin{theorem}\label{optimal-plane}
Let $\Omega\subset\mathbb C$ be a domain with $C^2$ smooth boundary, and assume that 
$\Omega$ is not simply connected.   Then there exists a constant 
$a>1$ such that 
\begin{equation}\label{optimal}
1-a\delta^2(z)\leq \frac{S_\Omega(z)}{K_\Omega(z)}\leq 1-\frac{1}{a}\delta^2(z),
\end{equation}
where $S$ and $K$ denote the Suita and Kobayashi metric respectively.  
\end{theorem}

Note that the theorem remains valid if we replace $S$ by any invariant metric dominated by $S$.

\medskip

The author would like to thank Nessim Sibony for invaluable suggestions and discussions on this subject.   

\section{Definitions}

We recall briefly the definitions of some invariant metrics.   For any domain $\Omega\subset\mathbb C^n$
we define the Kobayashi metric $K_\Omega(z,\xi)$
\begin{equation}
K_\Omega(z,\xi):=\inf\{\lambda:\exists f:\triangle\overset{\mathrm{hol}}{\rightarrow}\Omega,f(0)=z,\lambda f'(0)=\xi\},
\end{equation}
and the Carath\'{e}odory metric 
\begin{equation}
C_\Omega(z,\zeta):=\sup \{|f'(z)|:\exists f:\Omega\overset{\mathrm{hol}}{\rightarrow}\triangle, f(z)=0\}.
\end{equation}
So both are families of metrics defined for all domains, and it is easy to see that they are decreasing 
with respect to holomorphic maps; if $g:\Omega_1\rightarrow\Omega_2$ is holomorphic, then 
\begin{equation}\label{decreasing}
g^*F_{\Omega_2}\leq F_{\Omega_1},
\end{equation}
where $F$ is either of the two metrics.  On the unit disk $\triangle$ in the complex plane both metrics 
are given by $F(z)=\frac{1}{1-|z^2|}$.  More generally, a family of metrics $F_\Omega$
defined for all domains $\Omega\subset\mathbb C^n$ is called \emph{invariant} if $F_\triangle=K_\triangle=C_\triangle$
and if the decreasing property \eqref{decreasing} is satisfied.  It is not hard to see that for any invariant metric $F$
we have that
\begin{equation}
C_\Omega\leq F_\Omega\leq K_\Omega,
\end{equation}
for all domains $\Omega$.  \

A third explicit metric that we will consider is the Suita metric, defined on planar domains.  If $D\subset\mathbb C$ is 
a domain that supports a Green's function we do the following.  For a point $a\in D$, let 
$G_a(z)$ denote the negative Green's function with pole at $a$.   Then 
\begin{equation}
G_a(z)= \log |z-a| + h_a(z), 
\end{equation}
where $h_a$ is harmonic.  The Suita metric $c_\beta$ is defined by $c_\beta(a)=e^{h_a(a)}$. \\

A result similar to Theorem \ref{optimal-plane} was recently given by the author and Diederich-Forn\ae ss in \cite{DiederichFornaessWold2}.
We recall the definition of the squeezing function.  For a (bounded) domain $D\subset\mathbb C^n$ and an injective holomorphic map $f:D\rightarrow\mathbb B^n$ 
with $f(z)=0$ we set 
\begin{equation}
\mathcal S_{f}(z)=\sup\{r>0:r\mathbb B^n\subset f(D)\}.
\end{equation}
The squeezing function is then defined as 
\begin{equation}
\mathcal S(z)=\underset{f}{\sup}\{S_{f}(z)\}.
\end{equation}
We may think of the squeezing function as measuring how much the domain $D$ resembles the unit ball observed from the point $z$.  We can make 
the same interpretation of the quotient in the estimate \eqref{optimal}.  Merging the estimate \eqref{optimal} with the result in \cite{DiederichFornaessWold2}
we get the following equivalence expressing that a planar domain cannot look too much like the unit disk without being the unit disk.
\begin{theorem}\label{equivalence}
Let $D\subset\mathbb C$ be a $C^2$-smooth domain.  The following are equivalent.  
\begin{itemize}
\item[(i)] $D$ is biholomorphic to the unit disk, 
\item[(ii)] $\mathcal S(z)=1+ o(d(z,bD))$, and 
\item[(iii)] $\frac{S(z)}{K(z)}=1 + o(d(z,bD)^2)$,
\end{itemize}
where $\mathcal S$ denotes the squeezing function, and  $S$ and $K$ denote the Suita and the Kobayashi metric respectively. 
\end{theorem}

This therorem bares resemblance to gap phenomena in Riemannian geometry where euclidean spaces are characterised by 
asymptotic behaviour of curvatures of complete metrics, see \cite{GreeneWu}.

\section{Invariant metrics on planar domains}

In this section we will prove Theorem \ref{main-invariant} in the case of a planar domain.
The proof is elementary, but although the result is fairly well known, there does not seem 
to be a proof nor a statement of it in the literature.   So we include a detailed proof here (see however \cite{NikolovTrybulaAndreev}, Proposition 5,
for a similar estimate in the case of $C^{2+\epsilon}$-smooth boundary).  
The proof consists of squeezing our domain between real analytic topological disks, and 
elementary calculations involving Riemann mapping functions.  
Since the Riemann maps do not preserve the normal direction, we reformulate the result 
in terms of pinched cones. 

\begin{definition}\label{plane-invariant}
Let $D\subset\mathbb C$ be a domain with an outward pointing normal vector $e^{i\theta}$
at a point $p\in bD$.  For  $\eta,\nu>0$  we denote by $C_{\eta,\nu}(p)$ the set
\begin{equation*}
C_{\eta,\nu}(p)=\{z:|\mathrm{Im}(e^{-i\theta}(z-p))|<\eta\mathrm{Re}(e^{-i\theta}(z-p))^2,-\nu<\mathrm{Re}(e^{-i\theta}(z-p))<0 \},
\end{equation*}
and refer to this as a pinched cone for $D$ at $p$.    
\end{definition}

\begin{theorem}\label{plane-localise}
Let $D\subset\mathbb C$ be a domain of class $C^4$ and let $p\in bD$.  Then 
\begin{equation}\label{plane-estimate}
F_D(z) = \frac{1}{2|z-p|} + \frac{\kappa_p}{4} + O(|z-p|),
\end{equation}
for $z\in C_{\eta,\nu}(p)$ (for $\nu$ sufficiently small), where $F_D$ is any invariant metric on $D$, or 
the Bergman Kernel function on the diagonal.
\end{theorem}

By a translation and rotation we can always assume that $p=0$, and work with  
domains $D$ in $\mathbb C$ defined near the origin as $D=\{z:\phi(z)<0\}$ where 
\begin{equation}
\phi(z) = 2\mathrm{Re}(z + az^2) + b|z|^2 + O(|z|^3).
\end{equation}
The curvature $\kappa=\kappa(0)$ of $bD$ at the origin is defined as $b-2\mathrm{Re}(a)$.  Since 
another defining function for $D$ is $\tilde\phi(z)=(1-2\mathrm{Re}(az))\phi(z)$ we see that 
such a $D$ always has a defining function 
\begin{equation}\label{stand}
2\mathrm{Re}(z) + \kappa |z|^2 + O(|z|^3).
\end{equation}
It follows that if $\psi(z)=z + az^2 + O(|z|^3)$ then $\psi(D)$ has a defining function 
\begin{equation}
2\mathrm{Re}(z) + (\kappa + 2\mathrm{Re}(a))|z|^2 + O(|z|^3),
 \end{equation}
hence $bD$ has curvature $\kappa + 2\mathrm{Re}(a)$ at the origin.  So $a$ is purely imaginary 
if and only if $\psi$ preserves the curvature.  \\

\emph{Proof of Theorem \ref{plane-localise}:}
We may assume that $p=0$ and that $D$ is locally defined by \eqref{stand}, where $\kappa=\kappa_0$.
Let $D_1$ and $D_2$ be as in Lemma \ref{squeezecurvature}.   Let 
\begin{equation}
\psi(z)=z + O(|z|^2):D_1\rightarrow \tilde\triangle=\{z:2\mathrm{Re}(z) + |z|^2<0\}
\end{equation}
be the Riemann map, where $\tilde\triangle$ is the unit disk translated one unit to the left.   Then 
$\psi(z)=z + az^2 + O(|z|^2)$ where $1 = \kappa + 2\mathrm{Re}(a)$.
Using the explicit formula $\frac{1}{1-|z|^2}$ for the Kobayashi 
metric on the unit disk, the reader is asked to verify that Theorem \ref{plane-localise}
holds if $D$ is the unit disk, and so the theorem holds for $D_1$ by Lemma \ref{localchange}.
Similarly, the theorem holds for $D_2$, so the theorem is proved by the decreasing property 
of the metrics with respect to inclusions.  

$\hfill\square$

\begin{lemma}\label{localchange}
Let $D$ be a domain in $\mathbb C$ defined near the origin by $D=\{z:\phi(z)<0\}$ where 
\begin{equation}
\phi(z) = 2\mathrm{Re}(z) + \mathrm{h.o.t.}
\end{equation}
Let $\psi(z)=z+az^2+O(|z|^3)$, assume that $\psi$ maps $D$ biholomorphically onto a domain $\tilde D$
satisfying 
\begin{equation}
F_{\tilde D}(z)=\frac{1}{2|z|} + \frac{\kappa}{4} + O(|z|),
\end{equation}
in pinched cones at the origin.   Then 
\begin{equation}
F_{D}(z)=\frac{1}{2|z|} + \frac{\kappa-2\mathrm{Re}(a)}{4} + O(|z|),
\end{equation}
in pinched cones at the origin.

\end{lemma}
\begin{proof}
We estimate $\psi^*F_{\tilde D}(z)$.   Write $a=a_1+ia_2$.
Since we are in a pinched cone we have that $z=x + ig(x)$ with $g(x)=O(x^2)$.  So 
\begin{align*}
|\psi(z)|^2 & = (x + ig(x) + a(x + ig(x))^2 + O(|x|^3)) \\
& \cdot (\overline{x + ig(x) + a(x + ig(x))^2 + O(|x|^3)})\\
& =(x + ig(x) + (a_1 + ia_2)x^2 + O(|x|^3)) \\
& \cdot (x - ig(x) + (a_1- ia_2)x^2 + O(|x|^3))\\
& = x^2 + 2a_1x^3 +  O(|x|^4),
\end{align*}
and for $\psi'(z)=1+2az + O(|z|^2)$
\begin{align*}
|\psi'(z)|^2 &  = (1+2a(x+ig(x))+O(|x|^2))(\overline{1+2a(x+ig(x))+O(|x|^2)})\\
& = 1+ 4a_1x + O(|x|^2).
\end{align*}

So $|\psi(z)|=|x| - a_1x^2 + O(|x|^3)$ and $|\psi'(z)|=1 + 2a_1x + O(|x|^2)$.  We get that 
\begin{align*}
\psi^*F_{\psi(D)}(z) & = (\frac{1}{2|\psi(z)|} + \frac{\kappa}{4} + O(|z|))|\psi'(z)|\\
& =  (\frac{1}{2(|x| - a_1x^2 + O(|x|^3))} + \frac{\kappa}{4} + O(|z|))(1 + 2a_1x + O(|x|^2)) \\
& = (\frac{1}{2|x|(1+ a_1x + O(|x|^2))} + \frac{\kappa}{4} + O(|z|))(1+ 2a_1x+ O(|x|^2))\\
& = (\frac{1}{2|x|}(1 - a_1x +O(|x|^2) + \frac{\kappa}{4} + O(|z|))(1 + 2a_1x + O(|x|^2)))\\
& = \frac{1}{2|x|} - \frac{a_1x}{2|x|} + \frac{\kappa}{4} + \frac{a_1x}{|x|}+ O(|x|) \\
& =  \frac{1}{2|x|} + \frac{\kappa - 2a_1}{4} + O(|x|).
\end{align*}
Since $|z|=|x|+O(|x|^3)$ we get the desired estimate. 
\end{proof}

Since the local behaviour of invariant metrics depends only on the curvature, it is 
clear the the metrics localise correspondingly (for domains of class $C^4$).  
Localisation with the same estimate holds for domains of less smoothness; this is 
needed in the proof of Theorem \ref{optimal-plane}.

\begin{theorem}\label{localise-plane}
Let $D\subset\mathbb C$ be a domain of class $C^4$, and let $\tilde D\subset D$ 
be a domain with $p\in b\tilde D$ and $b\tilde D$ tangent to $bD$ to order two at $p\in bD$.  Then 
\begin{equation}\label{diff}
F_{\tilde D}(z) - F_{D}(z) = O(|z-p|),
\end{equation}
for $z\in C_{\eta,\nu}(p)$.  If $D$ is of class $C^2$ and $b\tilde D$ agrees with $bD$ on some open set containing $p$, 
then 
\begin{equation}
F_{\tilde D}(z) - F_{D}(z) = O(\delta(z)),
\end{equation}
locally near $p$, where $\delta$ denotes the boundary distance. 
\end{theorem}

\begin{proof}
The $C^4$-case is obvious in light of the previous discussion.   In the $C^2$-case we will show that 
\begin{equation}
g(z)=\frac{F_D(z)}{F_{\tilde D}(z)}\geq 1 - O(\delta^2(z))
\end{equation}
which gives the result.   By largening the domain we may assume that $D$ is simply connected.  
Let $\psi:D\rightarrow\triangle$ be a Riemann map with $\psi(p)=1$.
Then $\psi$ is $C^1$ up to the boundary and so it is enough to prove the estimate 
for $D_1=\psi(\tilde D)\subset\triangle$.  We could use that we are now in the $C^4$-case
and all that remains is to verify the estimate for metrics satisfying \eqref{plane-estimate}.  We will however give a
direct argument.  \

Choose a simply connected domain $D_2\subset D_1$ such that $bD_2$ agrees with $bD_1$ near the point one, and 
make sure that $D_2$ is symmetric with respect to the real axis.  Let $\psi:D_2\rightarrow \triangle$ be the Riemann 
map 
\begin{equation}
\psi(z)=z + a(z-1)^2 + O(|z-1|^3).
\end{equation}
symmetric with respect to the real axis.  So $\psi$ is real, and since the curvature is kept at the origin we
have $a=0$.  Calculating the metric on the real axis we use $r$ as a coordinate, and we have 
\begin{align*}
F_2(r)=\frac{\psi'(r)}{1-\psi(r)^2} & = \frac{1 + O((r-1)^2)}{1 - (r + O(|r-1|^3))^2}\\
& = \frac{1 + O((r-1)^2)}{(1-r^2)(1+O((r-1)^2))}\\
& = \frac{1}{1-r^2}[1 + O((r-1)^2)].
\end{align*}
This proves the estimate when we approach the boundary along the real axis.  
For other diameters close to this one one may simply slide $D_2$ along $bD_1$
and use the rotated Riemann map as before.   This makes it clear that there is 
a uniform bound on the remainder term.  
\end{proof}

\subsection{Local defining functions on planar domains}

For a $C^4$-smooth planar domain, we will near a boundary point obtain local coordinates in which there are particularly useful
defining functions.  The following is the main consequence needed here, but we will need such coordinates further in higher dimensions later.  

\begin{lemma}\label{squeezecurvature}
Let $D\subset\mathbb C$ be a domain of class $C^4$ and let $p\in bD$.   
Then there exist real analytic topological disks $D_1$ and $D_2$ with 
$D_1\subset D\subset D_2$, and $p$ is a boundary point of both 
$D_1$ and $D_2$ of the same curvature as $p$ in $bD$.
\end{lemma}

This follows from the following lemma. 

\begin{lemma}\label{defplane}
Let 
\begin{equation}
\rho(z)=2\mathrm{Re}(z) + \mathrm{h.o.t.}
\end{equation}
be a real valued function of class $C^4$ near
the origin in $\mathbb C$, and let $D=\{z:\rho(z)<0\}$.
Then for any $a\in\mathbb R$ there exists a holomorphic map 
\begin{equation}
\psi(z)=z + O(|z|^3),
\end{equation}
such that $D'=\psi(D)$ has a defining function 
\begin{equation}
\phi(z)=2\mathrm{Re}(z) + \kappa|z|^2 + a|z|^4 + o(|z|^4),
\end{equation}
where $\kappa$ is the curvature of $bD$ at the origin.
\end{lemma}

\begin{proof}
We have that
\begin{equation}
\rho(z)=z + \overline z + az^2 + \overline a\overline z^2 + b|z|^2 + O(|z|^3). 
\end{equation}
First we set 
\begin{equation}
\phi_1(z)=(1-az - \overline a\overline z)\rho(z)= z+\overline z + (b-2\mathrm{Re}(a))|z|^2 + O(|z|^3), 
\end{equation}
which we express as
\begin{equation}
\phi_1(z)=(1-az - \overline a\overline z)\rho(z)= z+\overline z + \kappa |z|^2 + P_3(z) + O(|z|^4). 
\end{equation}
where $P_3(z)=\sum_{\alpha+\beta=3}a_{\alpha\beta}z^\alpha\overline z^\beta$.  (Since $\phi_1$ is real we have that $a_{\alpha\beta}=\overline a_{\beta\alpha}$.)   If we have that 
\begin{equation}\label{pure}
P_3(z)=az^3 + \overline a\overline z^3, 
\end{equation}\label{pure}
we may set $\psi_1(z)=z - az^3$, and we get that $\phi_1(\psi_1(z))$ has 
no terms of degree three.   If $P_3$ contains a term $az\overline z^2 + \overline a\overline zz^2$
we set 
\begin{equation}
\phi_2(z)=(1-\overline az^2 - a\overline z^2)\phi_1(z),
\end{equation}
and we get that the degree three part of $\phi_2$ only has a term of the previous type.  
So we may assume that 
\begin{equation}
\phi_2(z)= z + \overline z + \kappa |z|^2 + P_4(z) + o(|z|^4).
\end{equation}
Now $P_4$ contains terms of type $z^\alpha\overline z^\beta + \overline z^\alpha z^\beta$ with $\alpha+\beta=4$.
By multiplying $\phi_2$ by a function of type $(1-z\overline z^2-\overline z z^2)$ we 
may produce the coefficient $a$ in front of the $|z|^4$-term, but we 
but we create a term
of type $z\overline z^3 + \overline zz^3$.  This term can be removed by multiplying with 
a function of type $1-z^3 - \overline z^3$, but we create a term of type $z^4 + \overline z^4$.
This can be removed by a coordinate change of type $\psi_2(z)= z + z^4$.
\end{proof}

\subsection{Proof of Theorem \ref{optimal-plane}}

It remains to prove the upper bound.   For this we will need an interpretation of the quotient $\frac{S_X(z)}{K_X(z)}$
on a Riemann surface $X$.  Fix $z\in X$, and let $[\gamma]$ be an element of $\pi_1(X)$ with base point $z$.
Then the class $[\gamma]$ has a unique element of least Kobayashi length; it has to be a straight line when lifted 
to the unit disk via the universal covering map.   We let $\Gamma_z(X)$ denote the collection of representatives 
of elements in $\pi_1(X)$ with base point $z$ with least length. 

\begin{lemma}
Let $X$ be a Riemann surface which is hyperbolic in the sense of Ahlfors.  Then 
\begin{equation}
\frac{S(z)}{K(z)}=\underset{\gamma\in\Gamma_z(X)}{\Pi}\frac{e^{2l_K(\gamma)}-1}{e^{2l_K(\gamma)}+1},
\end{equation}
where $l_K(\gamma)$ denotes the Kobayashi length of $\gamma$.
\end{lemma}
\begin{proof}
We fix a point $z\in X$, and we let $\pi:\triangle\rightarrow X$ be a universal covering map with $\pi(0)=z$.
We calculate the quotient in the local coordinates given by $\pi$.  In these coordinates the Kobayashi metric 
is equal to one.  By Myrberg's theorem (see \cite{Tsuji}, Theorem XI. 13.) the pullback of the Green's function with pole at $z$ is given by
\begin{equation}
G(\zeta)=\sum_{\phi}\log|\frac{\zeta-\phi(0)}{1-\overline{\phi(0)}\zeta}|
\end{equation}
where $\phi$ runs through the Deck-group.   So 
\begin{equation}
(G(z)-\log|z|)(0)=\sum_{\phi\neq\mathrm{id}}\log|\phi(0)|,
\end{equation}
and so the Suita metric at the origin is given by 
\begin{equation}
c_\beta(0)=\Pi_{\phi\neq\mathrm{id}}|\phi(0)|.
\end{equation}
The straight line segment between $0$ and $\phi(0)$ gives the shortest curve in the corresponding 
homotopy class, and its Kobayashi length is $\frac{1}{2}\log\frac{1+|\phi(0)|}{1-|\phi(0)|}$.
\end{proof}

To prove the upper bound we now fix a point $a$ close to a boundary point $p$.  The plan is for 
each $z$ on the straight line segment between $a$ and $p$, to chose 
elements $\gamma_{z,j}$ of the fundamental group of $D$ with base point $z$, and 
show that  
\begin{equation}
\underset{j}{\Pi}\frac{e^{2l_K(\gamma_{z,j})}-1}{e^{2l_K(\gamma_{z,j})}+1} =\underset{j}{\Pi} (1 - e^{-2l_K(\gamma_{z,j})}\frac{2}{1+ e^{-2l_K(\gamma_{z,j})}}) \leq 1 - C\delta^2(z)
\end{equation}
for $C>0$.
According to the lemma, these products are all larger than our quotient.  \

Fix all $\gamma_j\in\Gamma_a$.  Then $\gamma_{j,z}$ is defined by composing the straight path 
between $z$ and $a$, the loop $\gamma_j$ and the the straight path between $a$ and $z$.
The length of $\gamma_{j,z}$ is then the sum of the length of $\gamma_j$ and twice the 
length of the straight line segment between $a$ and $z$.  The length of the line segment between 
$a$ and $z$ is less than $\frac{1}{2}\log\frac{1}{\delta(z)} + \tilde C$.  So the total length of $\gamma_{j,z}$
is less than $l_K(\gamma_j)+\log\frac{1}{\delta(z)} + \tilde C$.   Write $a_j=e^{-2l_K(\gamma_j)}$.
So each factor in the 
product above is less than 
\begin{equation}
1 - a_j\delta^2(z)\tilde C,
\end{equation}
where the constant has changed.  Now 
\begin{equation}
-\log (1-y)=y(1+O(|y|)),
\end{equation}
and so 
\begin{equation*}
\frac{-\log (1-a_j\delta(z)^2\tilde C)}{-\log (1-a_j)}\geq \delta^2(z)\tilde C'\Rightarrow \log (1-a_j\delta(z)^2\tilde C)\leq \log (1-a_j)\delta^2(z)C',
\end{equation*}
and so 
\begin{equation}
\sum_j\log (1 - a_j\delta^2(z)\tilde C)\leq \delta^2(z)C'\sum_j a_j = \delta^2(z)(-b),
\end{equation}
for some constant $b>0$.  So 
\begin{equation}
\frac{S(z)}{K(z)}\leq e^{-b\delta^2(z)} = 1 - b\delta^2(z) + O(|b\delta^2(z)|^2)\leq 1 - C\delta^2(z).
\end{equation}
$\hfill\square$

\section{Several Complex Variables}

To prove Theorem \ref{main-localise} we are going to embed $\Omega$ onto a suitable domain $\tilde\Omega$, which is squeezed 
between a product domain (on the outside) and  a disk $D_\lambda$ (on the inside), and these three sets will all contain 
the image of $p$ in their (relative) boundaries, and have the same curvature $\kappa$ at that point; the presise statement 
is in Theorem \ref{expose} below.  The decreasing property of invariant metrics along 
with the one variable result will then enable us to prove Theorem \ref{main-localise}; this 
will be done in subsection \ref{ss1}.  To prove Theorem \ref{expose} we will need to 
find local coordinates near $p$ in which the boundary has an especially nice defining function; this
will done in subsection \ref{ss2}.  The embedding result will be proved in subsection \ref{ss3}.

\subsection{Proof of Theorem \ref{main-localise}}\label{ss1}

Instead of estimating the difference directly, we will estimate the quotient of the two metrics.  
We may  assume that $p=0$ and that we have  a defining function 
\begin{equation}
\phi(z)=2\mathrm{Re}(z_1) + \mathrm{h.o.t.}
\end{equation}
for $b\Omega$ near the origin.  Let $\psi$ be the map from Theorem \ref{expose} below.  
Set $p_\delta=-\delta\tilde\eta_p, q_\delta=\psi(p_\delta)$, 
$\tau_\delta=\psi_*\tilde\eta_\delta$ and $\tilde\tau_\delta=\tau_\delta/\|\tau_\delta\|$.
 
\begin{align*}
\frac{F_{\Omega\cap L_p}(p_\delta,\tilde\eta_\delta)}{F_{\Omega}(p_\delta,\tilde\eta_\delta)} & \leq \frac{F_{D_\lambda}(p_\delta,\tilde\eta_\delta)}{F_{\Omega}(p_\delta,\tilde\eta_\delta)}\\
&  = \frac{ F_{\psi(D_\lambda)}(q_\delta,\tau_\delta)}{F_{\psi(\Omega)}(q_\delta,\tau_\delta)}\\
& = \frac{ F_{\psi(D_\lambda)}(q_\delta,\tilde\tau_\delta)}{F_{\psi(\Omega)}(q_\delta,\tilde\tau_\delta)}\\
& \leq \frac{ F_{\tilde D_\lambda}(\pi_1(q_\delta),\pi_{1,*}\tilde\tau_\delta)}{F_{\tilde D\times\mathbb C^{n-1}}(q_\delta,\tilde\tau_\delta)}\\
&  \leq \frac{ F_{\tilde D_\lambda}(\pi_1(q_\delta),\pi_{1,*}\tilde\tau_\delta)}{F_{\tilde D}(\pi_1(q_\delta),\pi_{1,*}\tilde\tau_\delta)}\\
& = \frac{\frac{1}{2(\delta + O(\delta^2))}+\frac{\kappa}{4} + O_1(\delta) }{\frac{1}{2(\delta + O(\delta^2))}+\frac{\kappa}{4} + O_2(\delta)}\\
& = 1+O_3(\delta^2).
\end{align*}

We now get that 
\begin{align*}
F_{\Omega\cap L_p} - F_\Omega = F_\Omega\cdot (\frac{F_{\Omega\cap L_p}}{F_\Omega} - 1)=F_\Omega\cdot O(\delta^2)=O(\delta).
\end{align*}

%

\subsection{Local defining functions in SCV}\label{ss2}

\begin{lemma}\label{locdefscv}
Let 
\begin{equation}
\rho(z)=2\mathrm{Re}(z_1) + \mathrm{h.o.t.}
\end{equation}
be a real valued function of class $C^4$ near
the origin in $\mathbb C^n$, and assume that the hypersurface $\Sigma=\{z:\rho(z)=0\}$ is 
strictly pseudoconvex.  Then there exists an injective holomorphic map $G$ near the 
origin and a real valued function $\phi$ of class $C^4$ such that 
\begin{itemize}
\item[(i)] $G(z)=(z_1,Az') + O(\|z\|^2)$
\item[(ii)] $G(z_1,0)=(z_1 + O(|z_1|^3),0)$, 
\end{itemize}
and the hypersurface $\Sigma'=G(\Sigma)$ is defined by $\Sigma'=\{z:\phi(z)=0\}$, with
\begin{equation}
\phi(z)=2\mathrm{Re}(z_1) + \kappa |z_1|^2 + |z'|^2 + \tau\|z\|^4 + O(|z'|^2)O(\|z\|)+o(\|z\|^4)
\end{equation}
with $\tau>0$.
\end{lemma}
\begin{proof}
By Lemma \ref{defplane} we may assume that 
\begin{equation}
\phi(z_1,0)=z_1+\overline z_1 + \kappa |z_1|^2 + |z_1|^4 + o(|z_1|^4),
\end{equation}
and by a linear change of coordinates in the $z'$-variables we may diagonalise
the complex Hessian in these variables.  
So we may assume that 
\begin{align*}
\phi(z) & = z_1+\overline z_1 + \kappa |z_1|^2 + |z_1|^4 + \|z'\|^2\\
& + 2\mathrm{Re}(\sum_{ij} a_{ij}z_iz_j) + \sum_{j\geq 2}b_jz_1\overline z_j + \overline{b_j}\overline z_1 z_j\\
& + O(\|z\|^3),
\end{align*}
with no pure $z_1$-terms of degree three or four in the remainder, and $a_{11}=0$.   Multiplying $\phi$
by the function 
\begin{equation}
1 - \sum_{j\geq 2} b_j \overline z_j +\overline b_j z_j,
\end{equation}
we eliminate the terms $\sum_{j\geq 2}b_jz_1\overline z_j + \overline{b_j}\overline z_1 z_j$  but change 
the expression $2\mathrm{Re}(\sum_{ij} a_{ij}z_iz_j)$.  Applying a coordinate change 
\begin{equation}
\psi(z)=(z_1 - \sum_{ij} \tilde a_{ij}z_iz_j),z'), \tilde a_{11}=0,
\end{equation}
we get a defining function 
\begin{equation}
\phi(z)  = z_1+\overline z_1 + \kappa |z_1|^2 + |z_1|^4 + \|z'\|^2 + O(\|z\|^3).
\end{equation}
We proceed to remove unwanted terms of degree three; these are terms of 
degree one in the $z_j$-variables for $j=2,3,...,n$, \emph{i.e.}, terms of the form
\begin{equation}\label{terms}
a z_j z_1^k\overline z_1^l + \overline a \overline z_j \overline z_1^k z_1^l,
\end{equation}
with $k+l=2$.   We start by removing the terms where $l=2$; for each $j$ this is done by multiplying $\phi$
by 
\begin{equation}\label{removemult}
1 - (az_jz_1^{k}\overline z_1^{l-1} + \overline a\overline z_j\overline z_1^{k}z^{l-1}),
\end{equation}
and while we remove the term we want, we create a term \eqref{terms} with $k=l=1$.  This 
can in turn be removed by multiplying by a function \eqref{removemult} with $l=0$, 
and we are now left with a term \eqref{terms} with $k=2,l=0$.  This term is removed by 
a change of variables 
\begin{equation}\label{varchange}
\psi(z)=(z_1 - az_jz_1^k,z_2,\cdot\cdot\cdot,z_n),
\end{equation}
and only unwanted terms of degree at least four is created.  \

Following the same scheme, we may proceed to remove all terms \eqref{terms}
with $k+l=4$.  For each $j$, multiplying by a function \eqref{removemult} we remove a term
of type $(k,l)$, and we create a term of type $(k+1,l-1)$.  So we may proceed until 
we are left with a term of type $(k+l,0)$ which can be removed by a coordinate
change \eqref{varchange}. \

We now have a defining function 
\begin{equation}
\phi(z)=2\mathrm{Re}(z_1) + \kappa |z_1|^2 + \|z'\|^2 + |z_1|^4 + O(|z'|^2)O(\|z\|)+o(\|z\|^4).
\end{equation}
Applying the coordinate change 
\begin{equation}
\psi(z)=(z_1,z_2+\tilde\tau_2 z_2^2,...,z_n+\tilde\tau_n z_n^2),
\end{equation}
the expression $\sum_{j=1}^n |z_j|^4=(\sum_{j=1}^n |z_j|^2)^2 + O(\|z'\|^2)O(\|z\|)$ appears as 
wanted, and no other unwanted terms are created.  

\end{proof}

\subsection{Exposing a point with control of a defining function}\label{ss3}
We will consider domains $\Omega\subset\mathbb C^n$ with $0\in b\Omega$, 
at which the outward pointing normal is the $\mathrm{Re}(z_1)$-axis.  Recall that 
in this setting we have denoted the $z_1$-line by $L_0$.  For $\lambda>0$
we denote by $D_\lambda$ the intersection $\Omega\cap L_0\cap\mathbb B^n_\lambda$.

\begin{theorem}\label{expose}
Let $\Omega\subset\mathbb C^n$ be a domain such that $\overline\Omega$ is a Stein compact, 
Assume that $0\in b\Omega$ is a strictly pseudoconvex boundary point of class $C^4$, 
with a defining function 
\begin{equation}
\phi(z)=2\mathrm{Re}(z_1) + \mathrm{h.o.t.}
\end{equation}
near the origin. 
Then there exist a planar 
domain $\tilde D\subset L:=L_0$ and a holomorphic
embedding $\psi:\overline\Omega\rightarrow \overline {\tilde D}\times\mathbb C^{n-1}$ such 
that the following hold
\begin{itemize}
\item[(i)] $\psi(0)=0\in b_1\tilde D$ (where $b_1$ means the boundary relative to $L$), 
\item[(ii)] $\psi(z_1,0)=(z_1+O(|z_1|^3),O(|z_1|^2))$, and 
\item[(iii)] for $\lambda$ small enough there exists $\tilde D_\lambda\subset\tilde D$ such that 
$0\in b_1\tilde D_\lambda$, $\kappa:=\kappa_{\tilde D_\lambda}(0)=\kappa_{\tilde D}(0)$, and $\psi(D_\lambda)$
is a graph $(z_1,g(z_1))$ over $\tilde D_\lambda$, with $g(z_1)=O(|z_1|^2)$,
\end{itemize}
Also, the function $g$ is defined 
in an open neighbourhood of the origin in $L$.
\end{theorem}
\begin{proof}
The proof is based on that of the exposing of points in \cite{DiederichFornaessWold}, except that 
we need to take extra care of the local geometry at the exposed point.  \

Let $G$ be the map from Lemma \ref{locdefscv} above.  
By Anders\'{e}n-Lempert theory there exists $F_1\in\mathrm{Aut_{hol}}\mathbb C^n$ such that 
$F_1(z)=G(z)+O(\|z\|^5)$.  It follows that $\Omega':=F_1(\Omega)$ has a defining function 
\begin{equation}\label{deffunction}
\phi(z)=2\mathrm{Re}(z_1) + \kappa |z_1|^2 + \|z'\|^2 + \tau \|z\|^4 + O(\|z'\|^2)O(\|z\|)+o(\|z\|^4),
\end{equation}
with $\tau>0$ near the origin.   \

%

We next want to find $F_2\in\mathrm{Aut_{hol}}\mathbb C^n$ 
such that $F_2(\overline\Omega')$ does not intersect the positive real axis $\Gamma=\{z\in\mathbb C^n;x_1\geq 0, x_2=z_2=\cdot\cdot\cdot = z_n=0\}$
except for at the origin, and such that the map $F_2$  matches the identity to order four.   Such a map is furnished by the 
proof of (4) in the proof of Lemma 3.1 in \cite{DiederichFornaessWold}.
So 
the image $\Omega''=F_2(\Omega')$ still has a defining function of the form \eqref{deffunction}.  \

Before we proceed, note that we now have proved a "local" version of the theorem.   Since 
\begin{equation}
\|z'\|^2 + \tau\|z\|^4 + O(\|z'\|^2)O(\|z\|)+o(\|z\|^4)\geq 0,
\end{equation}
we see that, locally near the origin, we have that $\pi_1(\Omega_2)\subset D_\kappa\subset L$, and for small $\lambda$
we could set $\tilde D_\lambda:=\pi_1(F_2(F_1(D_\lambda)))$, and $\tilde D$ a suitable domain which coincides with $D_\kappa$
near the origin.   The final step is to construct a suitabe holomorphic injection that roughly stretches the boundary of $\Omega''$ along $\Gamma$,
to a point $(R,0,...,0)$ with $R>>1$.  \

Let $\epsilon>0$ be small.  For $0<\eta<<\epsilon$
we define 
\begin{equation}
A:= \{z\in\triangle^n_\epsilon: z\in\overline\Omega'', |z_1|\leq 2\eta\} 
\end{equation}
and
\begin{equation}
 B:=\{z\in\triangle^n_\epsilon:z\in\overline\Omega'',|z_1|\geq\eta\}\cup \overline\Omega''\setminus\triangle_\epsilon^n.
\end{equation}
Then $\Omega''=A\cup B$, and we claim that if $\epsilon$ and $\eta=\eta(\epsilon)$ are sufficiently small, then we have that $\overline{(A\setminus B)}\cap\overline{(B\setminus A)}=\emptyset$, 
and the set $C=A\cap B$ is a Stein compact.  
To see this, choose $\epsilon$ small enough such that
\begin{equation}
\Omega''\cap\triangle_\epsilon^n = \{z\in\triangle_\epsilon^n:\phi(z)<0\}.
\end{equation}
Then for small enough $\eta$ we have that $A\subset\triangle_{\epsilon/2}^n$,
because if $|z_1|\leq 2\eta$ and $\epsilon/2 \leq \|z'\|\leq\epsilon$,  then 
\begin{equation}
\phi(z)\geq -4\eta \pm 4\kappa\eta^2 + \epsilon^2/4 + O(\epsilon^3)>0
\end{equation}
for sufficiently small $\eta$.  Then 
\begin{align*}
\overline{(A\setminus B)}\cap\overline{(B\setminus A)} & = \overline{(A\setminus B)}\cap\overline{((B\cap\triangle_\epsilon^n)\cup(\overline\Omega''\setminus\triangle_\epsilon^n))\setminus A}\\
& =  \overline{(A\setminus B)}\cap\overline{(B\cap\triangle_\epsilon^n)\setminus A} = \emptyset.
\end{align*}
Furthermore, we have that 
\begin{equation}
C= \{z\in\overline\triangle^n_\epsilon:  |z_1|\leq 2\eta\}\cap  \{z\in\overline\triangle^n_\epsilon: |z_1|\geq \eta\}\cap\overline\Omega'',
\end{equation}
so $C$ is a Stein compact.  \

Next let $V=\Omega''\cap\triangle_\epsilon^n$.  Choose 
a smoothly bounded simply connected domain $V_1\subset L$ such that $bV_1$ has a defining function

\begin{equation}
2\mathrm{Re}(z_1) + \kappa |z_1|^2
\end{equation}
near the origin,   
such that 
$\overline V_1$ intersects $\Gamma$ only at the origin, and is symmetric with respect to the $x_1$-axis.  Note 
that, locally near the origin, $\pi_1(V)\subset V_1$, where $\pi$ denote the projection onto $L$. \

Next choose $R>0$ such that $\pi_1(\overline\Omega'')\subset\subset\triangle_R\subset L$, and set 
$\Gamma_R:=\Gamma\cap\{|z_1|\leq R\}$.  For large enough  $j$ and small enough $\nu<<1$, 
we let $W_j$ be the simply connected domain
\begin{equation}
W_j=(V_1\cup\Gamma_R(1/j))\setminus\{z_1:2\mathrm{Re}(z_1-R)+\kappa\|z_1-R\|^2\geq 0, |z_1-R|<\nu\},
\end{equation}
where $\Gamma_R(1/j)$ denotes the open $1/j$-neighbourhood of $\Gamma_R$.  Then $W_j$
is symmetric with respect to the $x_1$-axis, and $bW_j$, near $(R,0,...,0)$, is just a translate of $bV_1$ near the origin. \

Fix a point $a$ in $V_1$ intersected with the $x_1$-axis, and 
let $f_j:V_1\rightarrow W_j$ be the Riemann map such that $f_j(a)=a, f_j'(a)>0$.  By symmetry, $f_j$ is real, 
and $f_j$ sends the origin to $(R,0,...,0)$.  Then for any $\mu>0$ we have that 
$f_j\rightarrow\mathrm{id}$ uniformly on $\overline V_1\setminus D_\mu(0)$.  Set $\gamma_j(z)=(f_j(z_1),z_2,...,z_n)$, 
and note that $\gamma_j\rightarrow\mathrm{id}$ on a fixed open neighbourhood of $C=A\cap B$.  \

By Theorem 4.1 in \cite{Forstneric2003} there exist open neighbourhoods $A',B'$ and $C'$ of $A,B$ and $C$ respectively, 
and injective holomorphic maps $\alpha_j:A'\rightarrow\mathbb C^n, \beta_j:B'\rightarrow\mathbb C^n$, such 
that $\alpha_j\rightarrow\mathrm{id}, \beta_j\rightarrow\mathrm{id}$ uniformly as $j\rightarrow\infty$, and 
$\gamma_j=\beta_j\circ\alpha_j ^{-1}$ for each $j\in\mathbb N$.   Moreover, $\alpha_j$ vanishes to order four 
at the origin for all $j\in\mathbb N$ (we note that in the statement of Theorem 4.1 in \cite{Forstneric2003}, the assumption that $C$ is a Stein compact is missing, which 
seems necessary for the claim of  the varnishing order to hold).  \

Set $\psi_j:=\gamma_j\circ\alpha_j$ on $A'$ and $\psi_j:=\beta_j$ on $B'$.  Then $\alpha_j(A)$ has a defining function on the form \eqref{deffunction}, 
so for large $j$ we have that $\pi_1(\alpha_j(A))\subset V_1$.  So $\psi_j(A)\subset W_j\times\mathbb C^{n-1}$ for large $j$.  \

For a large $j$ let $\tilde D\subset L$ be a domain that agrees with $W_j$ near $(R,0,...,0)$, and such that $\pi_1(\Omega'')\subset W_j$.
For a small $\lambda>0$ set $\tilde D_\lambda:=\pi_1(\psi_j\circ F_2\circ F_1)(D_\lambda)$, and set $\psi=\psi_j\circ F_2\circ F_1 - (R,0,...,0)$.

\end{proof}

\bibliographystyle{amsplain}

\end{document}